\documentclass[11pt]{amsart}
\usepackage{cleveref,mathrsfs,enumitem, amssymb}
\usepackage{xcolor}%%eventually remove
\usepackage{tikz}
\usetikzlibrary{positioning}


\title{On nonmatrix varieties of associative rings}

\author{Thiago Castilho de Mello}
\address{Instituto de Ciência e Tecnologia, Universidade Federal de S\~ao Paulo, S\~ao Jos\'e dos Campos, SP, Brazil}
\email{tcmello@unifesp.br}
\author{Felipe Yukihide Yasumura}
\address{Department of Mathematics, Instituto de Matem\'atica, Estat\'istica e Ci\^encia da Computa\c c\~ao, Universidade de S\~ao Paulo, S\~ao Paulo, SP, Brazil}
\email{fyyasumura@ime.usp.br}

\thanks{This work is supported by S\~ao Paulo Research Foundation (FAPESP), grant 2025/02457-5 and 2024/14914-9, and by 
CNPq, grant 405779/2023-2.}

\renewcommand{\labelenumi}{(\roman{enumi})}
\newtheorem{Thm}{Theorem}
\newtheorem{Lemma}[Thm]{Lemma}
\newtheorem{Prop}[Thm]{Proposition}
\newtheorem{Cor}[Thm]{Corollary}
\theoremstyle{definition}
\newtheorem{Def}[Thm]{Definition}
\theoremstyle{remark}

\newtheorem{Remark}[Thm]{Remark}
\newtheorem{Example}[Thm]{Example}

\begin{document}
\begin{abstract}
We study nonmatrix varieties of $\mathbf{k}$-algebras, where $\mathbf{k}$ is a unital commutative ring. Our results extend to this generality known results for the case in which $\mathbf{k}$ is an infinite field. Also, we generalize these results to varieties of $\mathbf{k}$-algebras not containing the algebra of $n\times n$ matrices.

%We study varieties of $\mathbf{k}$-algebras that do not contain the algebra of $n\times n$ matrices over a field for a fixed integer $n$. This class includes all proper varieties of associative $\mathbf{k}$-algebras. Our results extend the classical theory of associative nonmatrix varieties (which concerns the case $n=2$) for arbitrary $n$, and for the broader setting in which $\mathbf{k}$ is a unital Noetherian commutative ring.
\end{abstract}
\maketitle

\section{Introduction}

Algebras satisfying polynomial identities share many properties with finite-dimensional and commutative algebras. For instance, there is a version of Hilbert's Nullstellensatz for PI-algebras \cite{Amitsur_Nullstellensatz}. This connection becomes even stronger when we restrict ourselves to algebras satisfying nonmatrix polynomial identities: even more results valid for commutative algebras hold for these algebras, for example, the sum of nilpotent elements is nilpotent. Thus, apparently, nonmatrix varieties are the varieties of associative non-commutative algebras that most closely resemble commutative ones.

Nonmatrix varieties appeared in the works by Latyshev \cite{LI,LII}. These varieties exhibit several noteworthy properties. Notably, Kemer explored the connection between nonmatrix varieties and the Specht property \cite{K79} before his groundbreaking results \cite{Kemerbook}. Additionally, Latyshev established that these algebras satisfy Hilbert's Basis Theorem \cite{L65}. In a related vein, Billig et al.~\cite{BRT97} identified a relationship between algebras satisfying nonmatrix polynomial identities and group identities, underscoring the rich structure of nonmatrix varieties. 

The works of \cite{C,MPR} have investigated equivalence formulations for nonmatrix varieties. We shall revisit their findings compellingly. Extending the concept to non-associative algebras, Shestakov and Bittencourt \cite{SB} (and also see \cite{Bit}) have broadened the definition of nonmatrix varieties for specific classes of non-associative algebras.

In this paper, we study the structure of nonmatrix varieties in greater generality, extending the theory to varieties of $\mathbf{k}$-algebras where $\mathbf{k}$ is a unital Noetherian commutative ring (\Cref{sec3}). We also generalize the construction to varieties that do not contain matrix algebras of fixed size $n$, and we obtain a characterization of these higher-degree nonmatrix varieties (\Cref{sec5}). This definition is related to the notion of a variety of \emph{complexity $n$}. For, we need to consider prime quotients of the algebras ruled by the multilinear polynomial identities of $n\times n$ matrices over $\mathbf{k}$ (\Cref{sec4}).

\section{Preliminaries}

We assume throughout that $\mathbf{k}$ is a unital commutative ring, unless stated otherwise.  
The following lemma is a key tool in the proof of our main theorem. It relies on Kaplansky’s theorem on primitive algebras satisfying a polynomial identity (see, for instance, \cite[Section 11.2.1]{AGPR}).

\begin{Lemma}\label{primitive_containsM_2}
Let $\mathscr{V}$ be a variety of associative $\mathbf{k}$-algebras satisfying a polynomial identity, where $\mathbf{k}$ is a unital commutative ring, and let $\mathcal{A}\in\mathscr{V}$ be a primitive algebra. Then either $\mathcal{A}$ is a field, or $\mathscr{V}$ contains $\mathrm{M}_2(S^{-1}\bar{\mathbf{k}})$, where $\bar{\mathbf{k}}=\mathbf{k}/I$ is a domain and $S=\bar{\mathbf{k}}\setminus\{0\}$.
\end{Lemma}

\begin{proof}
Let $\mathcal{A}\in\mathscr{V}$ be a primitive algebra. By Kaplansky’s theorem, there exists a division $\mathbf{k}$-algebra $D$ such that
\[
\mathcal{A}\cong \mathrm{M}_n(D)
\quad\text{and}\quad
\dim_{Z(D)}\mathcal{A}<\infty,
\]
where $Z(D)$ denotes the center of $D$.

Since $D$ is a $\mathbf{k}$-algebra, there is a ring homomorphism $\mathbf{k}\to D$. Let $\bar{\mathbf{k}}$ denote its image; then $\bar{\mathbf{k}}$ is a domain. Moreover, $\bar{\mathbf{k}}\subseteq Z(D)$, and hence the fraction field $\mathbb{F}$ of $\bar{\mathbf{k}}$ satisfies $\mathbb{F}\subseteq Z(D)$. It follows that
\[
\mathrm{M}_n(\mathbb{F})\subseteq \mathcal{A}.
\]
Consequently, either $\mathcal{A}=D$ or
\[
\mathrm{M}_2(\mathbb{F})\subseteq \mathrm{M}_n(\mathbb{F})\subseteq \mathcal{A}.
\]
In the latter case, we conclude that $\mathrm{M}_2(\mathbb{F})\in\mathscr{V}$.

Thus, we may assume that $\mathcal{A}=D$ is a division algebra. Let $Z=Z(D)$ be its center. If $D=Z$, then $\mathcal{A}$ is a field, and we are done. Assume therefore that $D\neq Z$.

If $Z$ is infinite, an argument analogous to the proof of \cite[Proposition~2.2.38]{AGPR} shows that
\[
\mathrm{Id}(D)=\mathrm{Id}(D\otimes_Z \overline{Z}),
\]
where $\overline{Z}$ is an algebraic closure of $Z$. The algebra $D\otimes_Z \overline{Z}$ is then a simple algebra over an algebraically closed field, and hence contains a copy of $\mathrm{M}_2(\mathbb{F})$. Therefore, $\mathrm{M}_2(\mathbb{F})\in\mathscr{V}$.

Finally, if $Z$ is finite, then the condition $\dim_Z D<\infty$ implies that $D$ itself is a field.
\end{proof}

% \textcolor{red}{
% \begin{Lemma}
%     Let $\mathcal{A}$ be an $\mathbb{F}$-algebra, let $K$ be a unital commutative $\mathbb{F}$-algebra and $\mathcal{R}$ be a commutative $\mathbb{F}$-algebra. Assume also that $\mathcal{A}$ and $\mathcal{R}$ are $K$-algebras. Let $f\in \mathrm{Id}_{\mathbb{F}}(\mathcal A)$ and assume $K$ contains an infinite field. Then $f\in \mathrm{Id}_K(\mathcal{A}\otimes_K\mathcal{R})$.
% \end{Lemma}
% \begin{proof}
%     Let $f\in\mathrm{Id}_{\mathbb{F}}(\mathcal{A})$. Then $f\in \mathrm{Id}_K(\mathcal A)$ (since $K$ is a unital $\mathbb{F}$-algebra). Now, \cite[Proposition~2.2.38]{AGPR} implies that $f\in \mathrm{Id}_K(\mathcal{A}\otimes_K\mathcal{R})$. But since $f\in \mathbb{F}\langle X \rangle$, it follows that $f\in \mathrm{Id}_\mathbb{F}(\mathcal{A}\otimes_K\mathcal{R})$.
% \end{proof}
% %
% }

More generally, an analogous result holds in the context of prime algebras. Recall that Posner's Theorem (see, for instance, \cite[Theorem 6.5]{Drensky_Formanek}) asserts that if $\mathcal{A}$ is a prime PI $\mathbf{k}$-algebra, then its center $Z(\mathcal{A})$ is nontrivial and the localization $S^{-1}\mathcal{A}$ is simple, where $S=Z(\mathcal{A})\setminus\{0\}$. If $Z(\mathcal{A})$ is finite, then it is a field and $S^{-1}\mathcal{A}=\mathcal{A}$. Otherwise, an argument similar to the classical one for algebras over infinite fields shows that the polynomial identities of $\mathcal{A}$ are generated by its multihomogeneous identities. Moreover, $\mathcal{A}$ and $S^{-1}\mathcal{A}$ satisfy the same set of multihomogeneous $\mathbf{k}$-polynomial identities. Consequently, as $\mathbf{k}$-algebras, $\mathcal{A}$ and $S^{-1}\mathcal{A}$ are PI-equivalent.

\begin{Lemma}\label{prime_containsM_2}
Let $\mathscr{V}$ be a variety of associative $\mathbf{k}$-algebras satisfying a polynomial identity, where $\mathbf{k}$ is a unital commutative ring, and let $\mathcal{A}\in\mathscr{V}$ be a prime algebra. Then either $\mathcal{A}$ is a commutative domain, or $\mathscr{V}$ contains $\mathrm{M}_2(S^{-1}\bar{\mathbf{k}})$, where $\bar{\mathbf{k}}=\mathbf{k}/I$ is a domain and $S=\bar{\mathbf{k}}\setminus\{0\}$.
\end{Lemma}

\begin{proof}
By Posner’s theorem, the localization $S^{-1}\mathcal{A}$ is a unital simple  $\mathbf{k}$-algebra satisfying the same polynomial identities as $\mathcal{A}$, where $S=Z(\mathcal{A})\setminus\{0\}$. In particular, it is a primitive algebra. Applying \Cref{primitive_containsM_2}, we conclude that either $S^{-1}\mathcal{A}$ is commutative or $\mathscr{V}$ contains $\mathrm{M}_2(\mathbb{F})$ for some field $\mathbb{F}$.
\end{proof}

For any $\mathbf{k}$-algebra $\mathcal{A}$, let $\mathrm{Nil}(\mathcal{A})$ denote the sum of all nil ideals of $\mathcal{A}$ (equivalently, the maximal nil ideal of $\mathcal{A}$), and let $\mathcal{B}(\mathcal{A})$ denote the intersection of all prime ideals of $\mathcal{A}$. In general, one has
\[
\mathcal{B}(\mathcal{A}) \subseteq \mathrm{Nil}(\mathcal{A}),
\]
and the two coincide when $\mathcal{A}$ is commutative. The set $\mathcal{B}(\mathcal{A})$ coincides with the set of all strongly nilpotent elements of $\mathcal{A}$.

We shall make use of the following result.

\begin{Lemma}[{\cite[Theorem 16.1.1]{AGPR}}]\label{nil_nilpotent} Let $\mathbf{k}$ be a unital Noetherian commutative ring and $\mathcal{A}$ a finitely generated $\mathbf{k}$-algebra satisfying a polynomial identity. Then, $\mathrm{Nil}(\mathcal{A})$ is nilpotent.
\end{Lemma}

Recall that a Jacobson ring is a ring in which every prime homomorphic image has trivial Jacobson radical. For example, fields and $\mathbb{Z}$ are Jacobson rings. With this terminology, the previous result can be reformulated as follows:

\begin{Lemma}[{\cite[Theorem 1.8]{KBR}}]\label{jacobsonnilpotent}
%%Theorem 1.8 de https://arxiv.org/pdf/1405.0730
Let $\mathbf{k}$ be a unital Jacobson Noetherian commutative ring and $\mathcal{A}$ a finitely generated PI-algebra over $\mathbf{k}$. Then, the Jacobson radical of $\mathcal{A}$ is nilpotent.
\end{Lemma}

\newpage\section{Nonmatrix variety over rings\label{sec3}}
We extend the characterization of nonmatrix varieties to the setting of $\mathbf{k}$-algebras, where $\mathbf{k}$ is a unital commutative ring.

\begin{Thm}\label{characterization_ringNoetherian}
Let $\mathbf{k}$ be a unital commutative ring, and let $\mathscr{V}$ be a variety of associative $\mathbf{k}$-algebras. The following conditions are equivalent:
\begin{enumerate}[label=(\roman{enumi})]
\item\label{I.def} Every simple $\mathbf{k}$-algebra in $\mathscr{V}$ is a field.
\item\label{I.M_2} $\mathrm{M}_2(\mathbb{F})\notin\mathscr{V}$ for every fraction field $\mathbb{F}$ of a prime homomorphic image of $\mathbf{k}$.
\item\label{I.subdirectJac} For every $\mathcal{A}\in\mathscr{V}$, the quotient $\mathcal{A}/J(\mathcal{A})$ is a subdirect product of fields.
\item\label{I.subdirect} For every $\mathcal{A}\in\mathscr{V}$, the quotient $\mathcal{A}/\mathcal{B}(\mathcal{A})$ is a subdirect product of commutative domains.
\item\label{I.powercommutator} There exists $m\in\mathbb{N}$ such that
\[
[x_1,x_2]^m \in \mathrm{Id}(\mathscr{V}).
\]
\item\label{I.jacobson} For every $\mathcal{A}\in\mathscr{V}$,
\[
\mathcal{B}(\mathcal{A})=\{a\in\mathcal{A}\mid \text{$a$ is nilpotent}\}.
\]
\item\label{I.jacobsonfg} For every finitely generated $\mathcal{A}\in\mathscr{V}$,
\[
\mathcal{B}(\mathcal{A})=\{a\in\mathcal{A}\mid \text{$a$ is nilpotent}\}.
\]
\item\label{I.someideal} For any $\mathcal{A}\in\mathscr{V}$ and any nilpotent element $a\in\mathcal{A}$, at least one of the following ideals is nil:
\[
a\mathcal{A},\quad a\mathcal{A}+\mathbf{k}a,\quad \mathcal{A}a,\quad \mathcal{A}a+\mathbf{k}a,\quad \mathcal{A}a\mathcal{A},\quad
\mathcal{A}a\mathcal{A}+\mathcal{A}a+a\mathcal{A}+\mathbf{k}a.
\]
\item\label{I.allideals} For any $\mathcal{A}\in\mathscr{V}$ and any nilpotent element $a\in\mathcal{A}$, all of the following ideals are nil:
\[
a\mathcal{A},\quad a\mathcal{A}+\mathbf{k}a,\quad \mathcal{A}a,\quad \mathcal{A}a+\mathbf{k}a,\quad \mathcal{A}a\mathcal{A},\quad
\mathcal{A}a\mathcal{A}+\mathcal{A}a+a\mathcal{A}+\mathbf{k}a.
\]
\item\label{I.sumnilelements} For every $\mathcal{A}\in\mathscr{V}$ and any nilpotent elements $a,b\in\mathcal{A}$, the element $a+b$ is nilpotent;
\item For every $\mathcal{A}\in\mathscr{V}$, any nilpotent element of $\mathcal{A}$ is strongly nilpotent;

%\item The variety $\mathscr{V}$ has complexity $0$ or $1$;

\item  For each $\mathcal{A}\in\mathscr{V}$ and nilpotent elements $a_1,\dots,a_m \in \mathcal{A}$, the subalgebra generated by $a_1,\dots,a_m$ is nilpotent.
\newcounter{contador1}
\setcounter{contador1}{\arabic{enumi}}
\end{enumerate}
\end{Thm}

\begin{proof}
The implications are summarized in the following diagram:

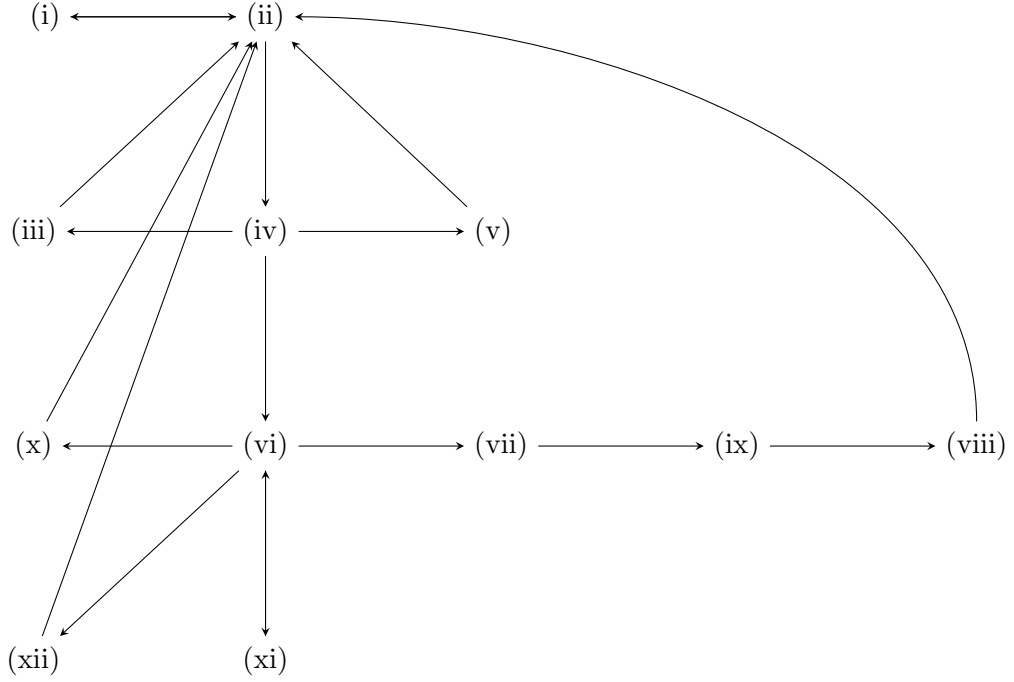
\begin{tikzpicture}[>=stealth, node distance=22mm]

% NODES
\node (i) {(i)};
\node (ii) [right=of i] {(ii)};
\node (iv) [below=of ii] {(iv)};
\node (iii) [left=of iv] {(iii)};
\node (v) [right=of iv] {(v)};

% terceira linha
\node (x) [below=of iii] {(x)};
\node (vi) [below=of iv] {(vi)};
\node (vii) [right=of vi] {(vii)};
\node (ix) [right=of vii] {(ix)};
\node (viii) [right=of ix] {(viii)};

% quarta linha
\node (xi) [below=of vi] {(xi)};
\node (xii) [below=of x] {(xii)};

% ARROWS
\draw[->] (i) -- (ii);
\draw[->] (ii) -- (i);

\draw[->] (ii) -- (iv);
\draw[->] (iv) -- (iii);
\draw[->] (iii) -- (ii);

\draw[->] (iv) -- (v);
\draw[->] (v) -- (ii);

\draw[->] (iv) -- (vi);
\draw[->] (vi) -- (vii);
\draw[->] (vii) -- (ix);
\draw[->] (ix) -- (viii);
\draw[->] (vi) -- (x);
\draw[->] (viii) .. controls +(0,40mm) and +(+40mm,0) .. (ii);

\draw[->] (x) -- (ii);

\draw[->] (vi) -- (xi);
\draw[->] (xi) -- (vi);

\draw[->] (vi) -- (xii);
\draw[->] (xii) -- (ii);

\end{tikzpicture}
\ref{I.def}$\Rightarrow$\ref{I.M_2}: This is immediate.

\medskip
\ref{I.M_2}$\Rightarrow$\ref{I.def}: 
Let $\mathcal{A}\in\mathscr{V}$ be a simple $\mathbf{k}$-algebra. By \Cref{prime_containsM_2}, either $\mathcal{A}$ is a commutative domain or $\mathscr{V}$ contains $\mathrm{M}_2(S^{-1}\bar{\mathbf{k}})$, where $\bar{\mathbf{k}}$ is a prime homomorphic image of $\mathbf{k}$ and $S=\bar{\mathbf{k}}\setminus\{0\}$. The latter is excluded by assumption. Hence, $\mathcal{A}$ is a simple commutative domain, and therefore a field.

\medskip
\ref{I.M_2}$\Rightarrow$\ref{I.subdirect}: 
It is well known that $\mathcal{A}/\mathcal{B}(\mathcal{A})$ is a subdirect product of prime algebras, each of which belongs to $\mathscr{V}$ and satisfies the same identities of its central localization, which is a simple algebra over its center. Since $\mathscr{V}$ does not contain $\mathrm{M}_2(\mathbb{F})$ for any field $\mathbb{F}$, \Cref{prime_containsM_2} implies that each of these prime algebras must be a commutative domain.

\medskip
\ref{I.subdirect}$\Rightarrow$\ref{I.powercommutator}: Consider the relatively free algebra $\mathbf{k}_{\mathscr{V}}\langle x_1,x_2\rangle$. Then, the quotient $\mathbf{k}_{\mathscr{V}}\langle x_1,x_2\rangle/\mathcal{B}(\mathbf{k}_{\mathscr{V}}\langle x_1,x_2\rangle)$ satisfies $[x_1,x_2]$. Since $\mathcal{B}(\mathbf{k}_{\mathscr{V}}\langle x_1,x_2\rangle$ is nil, there exists $m\in\mathbb{N}$ such that $[x_1,x_2]^m=0$. Therefore, $\mathscr{V}$ satisfies $[x_1,x_2]^m$.

\medskip
\ref{I.powercommutator}$\Rightarrow$\ref{I.M_2}: 
Observe that $[e_{12},e_{21}]^n\neq 0$ for every $n\in\mathbb{N}$. Thus, $\mathrm{M}_2(\mathbb{F})$ does not satisfy any power of the commutator polynomial, for any field $\mathbb{F}$. Hence, $\mathrm{M}_2(\mathbb{F})\notin\mathscr{V}$.

\medskip
\ref{I.subdirect}$\Rightarrow$\ref{I.jacobson}: 
Let $\mathcal{A}\in\mathscr{V}$ and let $a\in\mathcal{A}$ be nilpotent. Then $a+\mathcal{B}(\mathcal{A})$ is nilpotent in $\mathcal{A}/\mathcal{B}(\mathcal{A})$. Since $\mathcal{A}/\mathcal{B}(\mathcal{A})$ is a subdirect product of commutative domains, it follows that $a+\mathcal{B}(\mathcal{A})=0$. Hence, $a\in\mathcal{B}(\mathcal{A})$.

\medskip
\ref{I.jacobson}$\Rightarrow$\ref{I.jacobsonfg}: This is immediate.

\medskip
\ref{I.jacobsonfg}$\Rightarrow$\ref{I.allideals}: 
Let $\mathcal{A}\in\mathscr{V}$ and let $a\in\mathcal{A}$ be nilpotent. We show that $a\mathcal{A}$ is a nil ideal. It suffices to prove that $ab$ is nilpotent for every $b\in\mathcal{A}$. Set $\mathcal{A}_0=\mathrm{alg}(a,b)$. Then $\mathcal{A}_0\in\mathscr{V}$ is finitely generated. By assumption, $\mathcal{B}(\mathcal{A}_0)$ consists precisely of the nilpotent elements. Since $a\in\mathcal{B}(\mathcal{A}_0)$, it follows that $ab\in\mathcal{B}(\mathcal{A}_0)$ as well, and hence $ab$ is nilpotent. The same argument applies to all other listed ideals.

\medskip
\ref{I.allideals}$\Rightarrow$\ref{I.someideal}: This is immediate.

\medskip
\ref{I.someideal}$\Rightarrow$\ref{I.M_2}: 
The element $e_{12}\in\mathrm{M}_2(\mathbb{F})$ is nilpotent; however, none of the one-sided or two-sided ideals it generates is nil. Hence, $\mathrm{M}_2(\mathbb{F})\notin\mathscr{V}$.

\medskip
\ref{I.jacobson}$\Rightarrow$\ref{I.sumnilelements}: 
If $a,b\in\mathcal{A}$ are nilpotent, then $a,b\in\mathcal{B}(\mathcal{A})$. Since $\mathcal{B}(\mathcal{A})$ is an ideal, it follows that $a+b\in\mathcal{B}(\mathcal{A})$, and thus $a+b$ is nilpotent.

\medskip
\ref{I.sumnilelements}$\Rightarrow$\ref{I.M_2}: 
For any field $\mathbb{F}$, the algebra $\mathrm{M}_2(\mathbb{F})$ does not satisfy \ref{I.sumnilelements}, since $e_{12}$ and $e_{21}$ are nilpotent, whereas their sum is not.

\medskip
\ref{I.subdirect}$\Rightarrow$\ref{I.subdirectJac}: 
Since $\mathcal{B}(\mathcal{A})\subseteq J(\mathcal{A})$, the quotient $\mathcal{A}/J(\mathcal{A})$ is a homomorphic image of $\mathcal{A}/\mathcal{B}(\mathcal{A})$. In particular, $\mathcal{A}/J(\mathcal{A})$ is commutative. Therefore, it is a subdirect product of primitive commutative algebras, that is, fields.

\medskip
\ref{I.subdirectJac}$\Rightarrow$\ref{I.M_2}: This is immediate.

\medskip
(xi)$\iff$(vi)$\Rightarrow$(xii): clear.

\medskip
(xii)$\Rightarrow$(ii): the elements $e_{12}$ and $e_{21}$ are nilpotent; however they do not generate a nilpotent algebra in $\mathrm{M}_2(\mathbb{F})$.
\end{proof}

\begin{Remark}\label{remarknovo}
Consider the following statement:
\begin{enumerate}
\setcounter{enumi}{\arabic{contador1}}
\item\label{I.commutators} For every finitely generated $\mathcal{A}\in\mathscr{V}$, there exists $m\in\mathbb{N}$ such that
\[
[x_1,x_2]\cdots[x_{2m-1},x_{2m}] \in \mathrm{Id}(\mathcal{A}).
\]
\end{enumerate}
The above statement implies \ref{I.powercommutator} of the previous theorem: the relatively free algebra $\mathbf{k}_{\mathscr{V}}\langle x_1,x_2\rangle$ is finitely generated. Hence, it satisfies an identity of the form $[x_1,x_2]^m=0$ for some $m\in\mathbb{N}$. Therefore, $\mathscr{V}$ satisfies $[x_1,x_2]^m$.

Conversely, if, in addition, $\mathbf{k}$ is Noetherian, them \ref{I.subdirect} of the previous theorem implies the above statement: let $\mathcal{A}\in\mathscr{V}$ be finitely generated. By \Cref{nil_nilpotent}, the ideal $\mathrm{Nil}(\mathcal{A})$ is nilpotent; hence, $\mathrm{Nil}(\mathcal{A})^m=0$ for some $m\in\mathbb{N}$. Since $\mathcal{B}(\mathcal{A})\subseteq \mathrm{Nil}(\mathcal{A})$, it follows that $\mathcal{B}(\mathcal{A})^m=0$ as well. Moreover, $\mathcal{A}/\mathcal{B}(\mathcal{A})$ is a subdirect product of commutative algebras, and therefore satisfies the identity $[x_1,x_2]=0$. Consequently, the product of $m$ commutators is a polynomial identity for $\mathcal{A}$.
\end{Remark}
Condition~\emph{(ii)} of the above theorem motivates the terminology: we say that a variety $\mathscr{V}$ is \emph{nonmatrix} if it satisfies this condition. Accordingly, we say that a $\mathbf{k}$-algebra satisfies a \emph{nonmatrix polynomial identity} if it generates a nonmatrix variety.

Conditions~\emph{(vi)}--\emph{(x)} are satisfied by all commutative algebras. Hence, algebras satisfying a nonmatrix polynomial identity share several structural properties with commutative algebras. We single out the following consequence.

\begin{Cor}
Let $\mathbf{k}$ be a unital commutative ring, and let $\mathcal{A}$ be a $\mathbf{k}$-algebra satisfying a nonmatrix polynomial identity. Then
\[
\mathcal{B}(\mathcal{A})
=
\{\,a\in\mathcal{A}\mid \text{$a$ is nilpotent}\,\}.
\]\qed
\end{Cor}

\begin{Remark}
The converse of the preceding corollary does not hold in general. For example, although the nil radical of the real quaternion algebra $\mathbb{H}$ is trivial, $\mathbb{H}$ does not satisfy a nonmatrix polynomial identity. Indeed,
\[
\mathrm{Id}(\mathbb{H})=\mathrm{Id}\bigl(\mathrm{M}_2(\mathbb{R})\bigr).
\]
\end{Remark}
On the other hand, if $\mathbb{F}$ is algebraically closed and the algebra is finite dimensional, then the converse does hold.

\begin{Prop}
Let $\mathbb{F}$ be an algebraically closed field, and let $\mathcal{A}$ be a finite-dimensional $\mathbb{F}$-algebra such that
\[
\mathcal{B}(\mathcal{A})=\{\,a\in\mathcal{A}\mid \text{$a$ is nilpotent}\,\}.
\]
Then $\mathcal{A}$ generates a nonmatrix variety.
\end{Prop}

\begin{proof}
Since $\mathcal{A}$ is finite dimensional, $\mathcal{B}(\mathcal{A}) = J(\mathcal{A})$.
By the Artin--Wedderburn theorem, we have
\[
\mathcal{A}/J(\mathcal{A})
\cong
\mathrm{M}_{n_1}(\mathbb{F})\oplus\cdots\oplus\mathrm{M}_{n_t}(\mathbb{F}).
\]
If $n_j>1$ for some $j$, then one can find nilpotent elements $a,b\in\mathcal{A}$ whose sum is not nilpotent. This contradicts the hypothesis on $\mathcal{A}$. Hence, $\mathcal{A}/J(\mathcal{A})$ is commutative.

Since $J(\mathcal{A})$ is nilpotent, it follows that $\mathcal{A}$ satisfies a power of the commutator. Therefore, by condition~\ref{I.powercommutator} of the previous theorem, $\mathcal{A}$ generates a nonmatrix variety.
\end{proof}

It is worth emphasizing that finite dimensionality is essential in the proposition above. For example, the algebra
\[
\mathcal{A}:=\mathbb{F}e_{11}\oplus\mathbb{F}e_{22}\oplus \mathrm{M}_2(E'),
\]
where $E'$ denotes the nonunital Grassmann algebra, does not generate a nonmatrix variety, even though
\[
J(\mathcal{A})=\mathrm{M}_2(E')
=
\{\,a\in\mathcal{A}\mid \text{$a$ is nilpotent}\,\}.
\]

Specializing \Cref{characterization_ringNoetherian} to the case where $\mathbf{k}$ is a Jacobson Noetherian ring, we obtain the following characterization:

\begin{Thm}\label{characterization_ring}
Let $\mathbf{k}$ be a unital commutative Jacobson Noetherian ring, and let $\mathscr{V}$ be a variety of associative $\mathbf{k}$-algebras. The following conditions are equivalent:
\begin{enumerate}[label=(\roman{enumi})]
\item\label{ii.def} Every simple $\mathbf{k}$-algebra in $\mathscr{V}$ is a field;
\item\label{ii.M_2} $\mathrm{M}_2(\mathbb{F})\notin\mathscr{V}$ for each fraction field $\mathbb{F}$ of a prime homomorphic image of $\mathbf{k}$;
\item\label{ii.subdirectJac} For each $\mathcal{A}\in\mathscr{V}$, $\mathcal{A}/J(\mathcal{A})$ is a subdirect product of fields;
\item\label{ii.subdirect} For each $\mathcal{A}\in\mathscr{V}$, $\mathcal{A}/\mathcal{B}(\mathcal{A})$ is a subdirect product of commutative domains;
\item\label{ii.commutators} For every finitely generated $\mathcal{A}\in\mathscr{V}$, there exists $m\in\mathbb{N}$ such that
\[
[x_1,x_2]\cdots [x_{2m-1},x_{2m}] \in \mathrm{Id}(\mathcal{A});
\]
\item\label{ii.powercommutator} There exists $m\in\mathbb{N}$ such that $[x_1,x_2]^m \in \mathrm{Id}(\mathscr{V})$;
\item\label{ii.jacobson} For every $\mathcal{A}\in\mathscr{V}$,
\[
\mathcal{B}(\mathcal{A}) = \{\, a\in\mathcal{A} \mid \text{$a$ is nilpotent} \,\};
\]
\item\label{ii.jacobsonfg} For every finitely generated $\mathcal{A}\in\mathscr{V}$,
\[
J(\mathcal{A}) = \{\, a\in\mathcal{A} \mid \text{$a$ is nilpotent} \,\};
\]
\item\label{ii.someideal} For any $\mathcal{A}\in\mathscr{V}$ and nilpotent element $a\in\mathcal{A}$, at least one of the following ideals is nil: 
\[
a\mathcal{A},\quad a\mathcal{A}+\mathbf{k}a,\quad \mathcal{A}a,\quad \mathcal{A}a+\mathbf{k}a,\quad \mathcal{A}a\mathcal{A},\quad \mathcal{A}a\mathcal{A}+\mathcal{A}a+a\mathcal{A}+\mathbf{k}a;
\]
\item\label{ii.allideals} For any $\mathcal{A}\in\mathscr{V}$ and nilpotent element $a\in\mathcal{A}$, all of the above ideals are nil;
\item\label{ii.sumnilelements} For every $\mathcal{A}\in\mathscr{V}$ and nilpotent elements $a,b\in\mathcal{A}$, the sum $a+b$ is nilpotent;
\item For every $\mathcal{A}\in\mathscr{V}$, each nilpotent element is strongly nilpotent;
\item For every $\mathcal{A}\in\mathscr{V}$ and every nilpotent elements $a_1$, \dots, $a_m\in\mathcal{A}$, the subalgebra generated by $a_1$, \dots, $a_m$ is nilpotent.
\end{enumerate}
\end{Thm}

\begin{proof}
The only difference between the statements of this theorem and the statements of  \Cref{characterization_ringNoetherian} is in item  \ref{ii.jacobsonfg}, where $\mathcal{B}(\mathcal{A})$ is replaced by $J(\mathcal{A})$, but in this setting, for any finitely generated $\mathbf{k}$-algebra $\mathcal{A}$, \Cref{jacobsonnilpotent} implies
\[
\mathcal{B}(\mathcal{A}) = \mathrm{Nil}(\mathcal{A}) = J(\mathcal{A}),
\]
and the equivalences follow directly from \Cref{characterization_ringNoetherian} and \Cref{remarknovo}.
\end{proof}

For the particular case where $\mathbf{k}$ is a field  (finite or infinite), we have the following characterization.

\begin{Thm}\label{characterization}
Let $\mathbb{F}$ be an arbitrary field and $\mathscr{V}$ a variety of associative $\mathbb{F}$-algebras. The following conditions are equivalent:
\begin{enumerate}[label=(\roman{enumi})]
\item\label{i.def} Every simple algebra in $\mathscr{V}$ is a field;
\item\label{i.M_2} $\mathrm{M}_2(\mathbb{F})\notin\mathscr{V}$;
\item\label{i.subdirectJac} For each $\mathcal{A}\in\mathscr{V}$, $\mathcal{A}/J(\mathcal{A})$ is a subdirect product of fields;
\item\label{i.subdirect} For each $\mathcal{A}\in\mathscr{V}$, $\mathcal{A}/\mathcal{B}(\mathcal{A})$ is a subdirect product of commutative domains;
\item\label{i.commutators} For every finitely generated $\mathcal{A}\in\mathscr{V}$, there exists $m\in\mathbb{N}$ such that
\[
[x_1,x_2]\cdots [x_{2m-1},x_{2m}] \in \mathrm{Id}(\mathcal{A});
\]
\item\label{i.powercommutator} There exists $m\in\mathbb{N}$ such that $[x_1,x_2]^m \in \mathrm{Id}(\mathscr{V})$;
\item\label{i.jacobson} For every $\mathcal{A}\in\mathscr{V}$,
\[
\mathcal{B}(\mathcal{A}) = \{\, a\in\mathcal{A} \mid \text{$a$ is nilpotent} \,\}
\]
is an ideal of $\mathcal{A}$;
\item\label{i.jacobsonfg} For every finitely generated $\mathcal{A}\in\mathscr{V}$,
\[
J(\mathcal{A}) = \{\, a\in\mathcal{A} \mid \text{$a$ is nilpotent} \,\};
\]
\item\label{i.someideal} For any $\mathcal{A}\in\mathscr{V}$ and nilpotent element $a\in\mathcal{A}$, at least one of the following ideals is nil: 
\[
a\mathcal{A},\quad a\mathcal{A}+\mathbb{F}a,\quad \mathcal{A}a,\quad \mathcal{A}a+\mathbb{F}a,\quad \mathcal{A}a\mathcal{A},\quad \mathcal{A}a\mathcal{A}+\mathcal{A}a+a\mathcal{A}+\mathbb{F}a;
\]
\item\label{i.allideals} For any $\mathcal{A}\in\mathscr{V}$ and nilpotent element $a\in\mathcal{A}$, all of the above ideals are nil;
\item\label{i.sumnilelements} For every $\mathcal{A}\in\mathscr{V}$ and nilpotent elements $a,b\in\mathcal{A}$, the sum $a+b$ is nilpotent;
\item For every $\mathcal{A}\in\mathscr{V}$, each nilpotent element is strongly nilpotent;
\item For every $\mathcal{A}\in\mathscr{V}$ and every nilpotent elements $a_1$, \dots, $a_m\in\mathcal{A}$, the subalgebra generated by $a_1$, \dots, $a_m$ is nilpotent.\qed
\end{enumerate}
\end{Thm}

Imposing restrictions on the base field leads to additional equivalent properties.

\begin{Prop}
Let $\mathbb{F}$ be a field and $\mathscr{V}$ a variety of associative $\mathbb{F}$-algebras. If, for each $\mathcal{A}\in\mathscr{V}$, every $\mathbb{F}$-finite-dimensional irreducible representation of $\mathcal{A}$ is $1$-dimensional, then $\mathscr{V}$ is a nonmatrix variety. The converse holds if $\mathbb{F}$ is algebraically closed.
\end{Prop}

\begin{proof}
Assume every finite-dimensional irreducible representation of $\mathcal{A}\in\mathscr{V}$ is $1$-dimensional. Then $\mathcal{A}/J(\mathcal{A})$ is a subdirect product of copies of $\mathbb{F}$ and possibly infinite-dimensional $\mathbb{F}$-primitive algebras. If any of the primitive algebras is noncommutative, then \Cref{primitive_containsM_2} implies that $\mathscr{V}$ contains a copy of $\mathrm{M}_2(\mathbb{F})$, giving a $2$-dimensional representation, a contradiction. Hence, $\mathcal{A}/J(\mathcal{A})$ is a subdirect product of fields, and $\mathscr{V}$ is a nonmatrix variety.

Conversely, if $\mathbb{F}$ is algebraically closed, then for any $\mathcal{A}\in\mathscr{V}$, the quotient $\mathcal{A}/J(\mathcal{A})$ is a subdirect product of fields. Therefore, any irreducible representation of $\mathcal{A}$ corresponds to a field extension $\mathbb{E}$ of $\mathbb{F}$. Since $\mathbb{F}$ is algebraically closed, either $\mathbb{E} = \mathbb{F}$ or $\dim_\mathbb{F} \mathbb{E} = \infty$.
\end{proof}

\begin{Example}
Let $\mathbb{F}$ be a field admitting an $n$-dimensional extension $\mathbb{E}/\mathbb{F}$, with $n>1$. Consider the nonmatrix variety $\mathscr{V}$ of $\mathbb{F}$-algebras defined by $[x,y] = 0$. Then $\mathbb{E} \in \mathscr{V}$, but it has an irreducible representation of $\mathbb{F}$-dimension $n$. This shows that the assumption of $\mathbb{F}$ being algebraically closed is essential.
\end{Example}

% \begin{Prop}
% Let $\mathbf{k}$ be a unital commutative ring and $\mathscr{V}$ a variety of associative $\mathbf{k}$-algebras. If $\mathscr{V}$ is a nonmatrix variety, then for each $\mathcal{A}\in\mathscr{V}$ and nilpotent elements $a_1,\dots,a_m \in \mathcal{A}$, the subalgebra generated by $a_1,\dots,a_m$ is nilpotent. The converse holds if there is a infinite prime homomorphic image of $\mathbf{k}$.
% \end{Prop}
% \begin{proof}
% \end{proof}

\begin{Prop}\label{prop}
Let $\mathbf{k}$ be a unital commutative ring and $\mathscr{V}$ a variety of associative $\mathbf{k}$-algebras. If $\mathscr{V}$ is a nonmatrix variety, then for each $\mathcal{A}\in\mathscr{V}$ and integral elements $a_1,\dots,a_m \in \mathcal{A}$, the subalgebra generated by $a_1,\dots,a_m$ is finitely generated as a $\mathbf{k}$-module. The converse holds if there is an infinite prime homomorphic image of $\mathbf{k}$.
\end{Prop}

\begin{proof}
Assume $\mathscr{V}$ is a nonmatrix variety. Let $\mathcal{A}\in\mathscr{V}$ and $a_1,\dots,a_m \in \mathcal{A}$ be integral. Set $\mathcal{A}_0 = \mathrm{alg}(a_1,\dots,a_m)$. By \Cref{characterization_ringNoetherian}.\ref{I.subdirect}, $\mathcal{A}_0 / \mathcal{B}(\mathcal{A}_0)$ is commutative. Since it is finitely generated by integral elements, it is $\mathbf{k}$-integral. Since $\mathcal{B}(\mathcal{A}_0)$ is nil, $\mathcal{A}_0$ is $\mathbf{k}$-integral. By Kaplansky Theorem (see, for instance, \cite[Theorem 8.2.1]{AGPR}), an integral PI $\mathbf{k}$-algebra is locally finitely generated, so $\mathcal{A}_0$ is finitely generated as a $\mathbf{k}$-module.

Conversely, let $\mathbb{F}$ be the fraction field of an infinite prime homomorphic image of $\mathbf{k}$. If $\mathrm{M}_2(\mathbb{F}) \in \mathscr{V}$, then, as $\mathbb{F}$ is infinite, $\mathrm{M}_2(\mathbb{F}[t]) \cong\mathrm{M}_2(\mathbb{F})\otimes \mathbb{F}[t] \in \mathscr{V}$ as well. The elements $\xi_1 = t e_{12}$ and $\xi_2 = t e_{21}$ are integral (in fact, nilpotent), yet they generate an infinite-dimensional $\mathbb{F}$-algebra, hence not finitely generated as a $\mathbf{k}$-module.
\end{proof}

For the particular case where $\mathbf{k}$ is a field, we recover a classical result (see, for instance, \cite[Proposition 2]{C}):

\begin{Prop}
Let $\mathbb{F}$ be an arbitrary field and $\mathscr{V}$ a variety of associative $\mathbb{F}$-algebras. If $\mathscr{V}$ is a nonmatrix variety, then for each $\mathcal{A}\in\mathscr{V}$ and algebraic elements $a_1, \dots, a_m \in \mathcal{A}$, the subalgebra generated by $a_1, \dots, a_m$ is finite-dimensional. The converse holds if $\mathbb{F}$ is infinite. \qed
\end{Prop}

\begin{Example}
Let $\mathbb{F}$ be a finite field and consider $\mathscr{V} = \mathrm{var}(\mathrm{M}_2(\mathbb{F}))$. This variety is locally finite since it is generated by a finite algebra. Hence, every set of algebraic elements generates a finite-dimensional algebra. However, $\mathrm{var}(\mathrm{M}_2(\mathbb{F}))$ is not a nonmatrix variety. This shows that the assumption of the base field being infinite is essential in the previous proposition.
\end{Example}

\begin{Remark}\label{remark16}
For a given variety $\mathscr{V}$ of associative $\mathbf{k}$-algebras, the following statements are equivalent:
\begin{enumerate}
\renewcommand{\labelenumi}{(\alph{enumi})}
\item There exists $m \in \mathbb{N}$ such that $\mathscr{V}$ satisfies the polynomial identity
\[
[x_1, x_2] \cdots [x_{2m-1}, x_{2m}];
\]
\item For each $\mathcal{A} \in \mathscr{V}$, the ideal generated by $[\mathcal{A}, \mathcal{A}]$ is nilpotent.
\end{enumerate}

Moreover, either of the above conditions implies that the Grassmann algebra $E$ does not belong to $\mathscr{V}$. It is known that $\mathscr{V}$ is generated by a finitely generated algebra if and only if $G\notin\mathscr{V}$ (see, for instance, \cite[Corollary 20.1.4]{AGPR}). Hence, if $\mathscr{V}$ is generated by a finitely generated algebra, then both conditions hold.  

In addition, over a field of characteristic $0$, it is known that $E \notin \mathscr{V}$ implies that $\mathscr{V}$ is generated by a finitely generated algebra. Therefore, we recover the known equivalence of conditions characterizing nonmatrix varieties in terms of the exclusion of the Grassmann algebra (\cite[Theorem 6]{MPR}).
\end{Remark}

\section{Nonmatricial radical\label{sec4}}
In this section, unless otherwise stated, we shall assume that $\mathbf{k}$ is a unital commutative ring.

Item~\ref{ii.powercommutator} of \Cref{characterization_ring} studies the varieties of algebras satisfying a power of the standard polynomial of degree $2$. It is known that every PI-ring satisfies a power of some standard polynomial. We shall construct a way to measure if an element of the algebra lies in some matrix algebra.

Let $\mathcal{M}$ be an irreducible $\mathcal{A}$-module. We define
\[
d(\mathcal{M}) = \sqrt{\dim_{Z(\mathcal{D})} \mathcal{D}} \;\cdot\; \dim_\mathcal{D} \mathcal{M},
\]
where $\mathcal{D} := \mathrm{End}_\mathcal{A}(\mathcal{M})$).

\begin{Def}\label{n-nilpotency}
Let $\mathcal{A}$ be a $\mathbf{k}$-algebra and let $n \ge 1$. An element $a \in \mathcal{A}$ is called \emph{$n$-nonmatricial} if it annihilates every irreducible $Q(\mathcal{A}_0)$-module $\mathcal{M}$ such that $d(\mathcal{M}) \le n$, where $\mathcal{A}_0$ is a prime homomorphic image of $\mathcal{A}$, $Q(\mathcal{A}_0) = C(\mathcal{A}_0) \mathcal{A}_0$ is the central closure of $\mathcal{A}_0$, and $C(\mathcal{A}_0)$ is its extended centroid.
\end{Def}

The definition of $n$-nonmatricial element measures if the element cannot live in some matrix algebra of order $\le n$. The following examples and propositions will shed light in this concept.

It is known that elements of the Jacobson radical can be characterized in terms of annihilators of irreducible modules. Similarly, $n$-nonmatricial elements can be characterized as the annihilators of modules as in \Cref{n-nilpotency}. In particular, we have the following:

\begin{Lemma}\label{module_lemma}
Let $\mathcal A$ be a $\mathbf k$-algebra. For any prime homomorphic image $\mathcal A_0$ of $\mathcal A$ and any irreducible $Q(\mathcal A_0)$-module $\mathcal M$ with $d(\mathcal{M}) < \infty$, the annihilator $\mathrm{Ann}_{\mathcal A}(\mathcal M)$ is a prime ideal of $\mathcal{A}$. Moreover, if $\mathcal{A}$ is PI, then every prime ideal arises in this way.
\end{Lemma}

\begin{proof}
Assume that $\mathcal A_0$ is a prime homomorphic image of $\mathcal A$, and that $\mathcal M$ is an irreducible $Q(\mathcal A_0)$-module with $d(\mathcal{M}) < \infty$. Then the quotient $Q(\mathcal{A}_0)/\mathrm{Ann}_{Q(\mathcal{A}_0)}(\mathcal{M})$ is simple and finite-dimensional over its center. We show that $\mathrm{Ann}_{\mathcal{A}_0}(\mathcal{M})$ is a prime ideal of $\mathcal{A}_0$.

Let $a, b \in \mathcal{A}_0$ be such that $a\mathcal{A}_0 b \subseteq \mathrm{Ann}_{\mathcal{A}_0}(\mathcal{M})$. Then
\begin{align*}
a Q(\mathcal{A}_0) b &= a C(\mathcal{A}_0) \mathcal{A}_0 b = C(\mathcal{A}_0) a \mathcal{A}_0 b \subseteq C(\mathcal{A}_0) \mathrm{Ann}_{\mathcal{A}_0}(\mathcal{M}) \\
&\subseteq \mathrm{Ann}_{Q(\mathcal{A}_0)}(\mathcal{M}).
\end{align*}
Since $Q(\mathcal{A}_0)/\mathrm{Ann}_{Q(\mathcal{A}_0)}(\mathcal{M})$ is simple (and hence prime), it follows that either $a$ or $b$ belongs to $\mathrm{Ann}_{Q(\mathcal{A}_0)}(\mathcal{M})$. Therefore, $\mathrm{Ann}_{\mathcal{A}_0}(\mathcal{M})$ is a prime ideal of $\mathcal{A}_0$, and thus
\[
\mathrm{Ann}_\mathcal{A}(\mathcal{M}) = \mathrm{Ker}(\mathcal{A} \to \mathcal{A}_0 / \mathrm{Ann}_{\mathcal{A}_0}(\mathcal{M}))
\]
is a prime ideal of $\mathcal{A}$.

Now, assume that $\mathcal{A}$ is PI and let $I \subseteq \mathcal{A}$ be a prime ideal. By Posner's Theorem, $Q(\mathcal{A}/I)$ is a simple algebra of finite dimension over its center. Let $\mathcal{M}$ be an irreducible $Q(\mathcal{A}/I)$-module. Then $d(\mathcal{M}) < \infty$, and $\mathcal{M}$ is faithful as a $Q(\mathcal{A}/I)$-module. Hence, $I = \mathrm{Ann}_\mathcal{A}(\mathcal{M})$.
\end{proof}

\begin{Def}
Let $\mathcal{A}$ be a $\mathbf{k}$-algebra and $w\in\mathbb{N}$. We define the \emph{$w$-nonmatricial ideal} as
\[
\mathfrak{M}_w(\mathcal{A}) := \{\text{$w$-nonmatricial elements of $\mathcal{A}$}\}.
\]
In addition, we define
\[
\mathfrak{M}_\infty(\mathcal{A}) := \bigcap_{w \in \mathbb{N}} \mathfrak{M}_w(\mathcal{A}).
\]
\end{Def}
It is not hard to see that each $\mathfrak{M}_w(\mathcal{A})$, for $w\in\mathbb{N}\cup\{\infty\}$, is a complete Hoehnke radical (see, for instance, \cite[Section 2.1]{GWbook}). Moreover, we have a decreasing sequence of ideals:
\[
\mathfrak{M}_1(\mathcal{A}) \supseteq \mathfrak{M}_2(\mathcal{A}) \supseteq \cdots
\]

\begin{Example}
Let
\begin{align*}
\mathcal{A}&=\mathrm{UBT}(1,2,3;\mathbb{F})=\left\{\begin{pmatrix}A_{11}&A_{12}&A_{13}\\0&A_{22}&A_{23}\\0&0&A_{33}\end{pmatrix}\,\bigg|\,A_{ij}\in\mathrm{M}_{i\times j}(\mathbb{F})\right\}\\%
&=\mathbb{F}\oplus\mathrm{M}_2(\mathbb{F})\oplus\mathrm{M}_3(\mathbb{F})\oplus J.
\end{align*}
Then, it is not hard to see that
\begin{equation*}
\begin{split}
&\mathfrak{M}_1(\mathcal{A})=\mathrm{M}_2(\mathbb{F})\oplus\mathrm{M}_3(\mathbb{F})\oplus J,\quad\mathfrak{M}_2(\mathcal{A})=\mathrm{M}_3(\mathbb{F})\oplus J,\\%
&\mathfrak{M}_\infty(\mathcal{A})=\mathfrak{M}_k(\mathcal{A})=J,\quad\forall k\ge3.
\end{split}
\end{equation*}
\end{Example}

It is clear that $\mathfrak{M}_\infty(\mathcal{A})$ coincides with the intersection of all annihilators of modules as in the statement of \Cref{module_lemma}. Moreover, the same lemma implies the following:

\begin{Prop}\label{prop19}
Let $\mathbf k$ be a unital commutative ring, and let $\mathcal A$ be a $\mathbf k$-algebra. Then
\[
\mathcal{B}(\mathcal{A}) \subseteq \mathfrak{M}_\infty(\mathcal{A}).
\]
In addition, if $\mathcal{A}$ is PI, then $\mathcal{B}(\mathcal{A}) = \mathfrak{M}_\infty(\mathcal{A})$.\qed
\end{Prop}

The following example shows that the equality $\mathfrak{M}_\infty(\mathcal{R})=\mathcal{B}(\mathcal{R})$ might hold even for non-PI rings.
\begin{Example}\label{UT}
Let 
\[
\mathrm{UT} = \bigcup_{n \in \mathbb{N}} \mathrm{UT}_n(\mathbb{F})
\] 
be the algebra of infinite upper triangular matrices with only finitely many nonzero entries. It is clear that $\mathrm{UT}$ is not PI; however, 
\[
\mathcal{B}(\mathrm{UT}) = \mathfrak{M}_1(\mathrm{UT})=\mathfrak{M}_\infty(\mathrm{UT}).
\] 
\end{Example}

The next example shows that $\mathfrak{M}_\infty(\mathcal{R})$ does not coincide with other known radicals.
\begin{Example}
Let $\mathcal{R}$ be a unital simple non-Artinian ring. Then $\mathcal{R}$ is centrally closed and admits no irreducible module $\mathcal{M}$ satisfying $d(\mathcal{M}) < \infty$. Consequently, every element $r \in \mathcal{R}$ is $n$-nonmatricial for every $n \in \mathbb{N}$. Hence,
\[
\mathfrak{M}_\infty(\mathcal{R}) = \mathcal{R}.
\]
In particular, $\mathfrak{M}_\infty(\mathcal{R})$ does not coincide with any of the radicals $\mathcal{B}(\mathcal{R})$, $\mathrm{Nil}(\mathcal{R})$, or $J(\mathcal{R})$.
\end{Example}

The next example shows that the sequence $(\mathfrak{M}_n(\mathcal{A}))$ does not necessarily stabilize.

\begin{Example}
Let us consider the algebra of block triangular matrices with infinite increasing sized blocks given by
\begin{align*}
\mathrm{UBT} &= \bigcup_{n \in \mathbb{N}} \mathrm{UBT}(1,2,\ldots,n) \\
&= \left\{\begin{pmatrix}
A_{11} & A_{12} & \cdots & A_{1m} & 0 & \cdots \\
0 & A_{22} & \cdots & A_{2m} & 0 & \cdots \\
\vdots & \ddots & \ddots & \vdots & \vdots & \cdots \\
0 & & \ddots & A_{mm} & 0 & \cdots \\
0 & 0 & \cdots & 0 & 0 & \cdots \\
\vdots & \vdots & \vdots & \vdots & \vdots & \ddots
\end{pmatrix}\,\, \Bigg|\,\,m \in \mathbb{N},\; A_{ij} \in \mathrm{M}_{i \times j}(\mathbb{F})\right\}.
\end{align*}
Then, $\mathrm{UBT}$ admits matrix algebras of any order as a homomorphic image. In addition, the sequence
\[
\mathfrak{M}_1(\mathrm{UBT}) \supsetneq \mathfrak{M}_2(\mathrm{UBT}) \supsetneq \cdots
\]
does not stabilize. Moreover, $\mathcal{B}(\mathrm{UBT}) = \mathfrak{M}_\infty(\mathrm{UBT})$.
\end{Example}

The following example clarifies the connection between the definition of a nonmatricial element and the polynomial identities satisfied by the algebra.
\begin{Example}\label{obs}
Let $\mathcal{A}$ be a PI $\mathbf{k}$-algebra, $\mathcal{A}_0$ a prime homomorphic image of $\mathcal{A}$, and $Q(\mathcal{A}_0)$ its central closure. Suppose there exists an irreducible $Q(\mathcal{A}_0)$-module $\mathcal{M}$ with $d(\mathcal{M}) = n$. Then, for some field of fractions $\mathbb{F}$ of a prime homomorphic image of $\mathbf{k}$, we have
\[
\mathrm{M}_n(\mathbb{F}) \in \mathrm{var}(\mathcal{A}).
\]

Indeed, by Posner's Theorem, $Q(\mathcal{A}_0) = Z^{-1}\mathcal{A}_0$ is simple and finite-dimensional over its center $Z = Z(\mathcal{A}_0)$. If $Z$ is infinite, then $\mathrm{Id}(\mathcal{A}_0) = \mathrm{Id}(Q(\mathcal{A}_0))$. Otherwise, $Z$ is a field and $Q(\mathcal{A}_0) = \mathcal{A}_0$. Therefore,
\[
Q(\mathcal{A}_0) \in \mathrm{var}(\mathcal{A}_0) \subseteq \mathrm{var}(\mathcal{A}).
\]
Moreover, it contains the image of the map
\[
Q(\mathcal{A}_0) \to \mathrm{End}_{\mathcal{D}}(\mathcal{M}),
\]
where $\mathcal{D} = \mathrm{End}_{Q(\mathcal{A}_0)}(\mathcal{M})$. Since $\dim_\mathcal{D} \mathcal{M} < \infty$, Density's Theorem ensures that the map is surjective. Hence, $\mathrm{var}(\mathcal{A})$ contains $\mathrm{End}_\mathcal{D}(\mathcal{M}) \cong \mathrm{M}_k(\mathcal{D})$, where $k = \dim_\mathcal{D} \mathcal{M}$. It follows that
\[
\mathrm{M}_{kr}(Z(\mathcal{D})) \in \mathrm{var}(\mathcal{A}),
\]
where $r = \sqrt{\dim_{Z(\mathcal{D})} \mathcal{D}}$. Finally, since $Z(\mathcal{D})$ is a $\mathbf{k}$-algebra, it contains a prime homomorphic image $\hat{\mathbf{k}}$ of $\mathbf{k}$. Therefore,
\[
\mathrm{M}_{kr}(\mathbb{F}) \in \mathrm{var}(\mathcal{A}),
\]
where $\mathbb{F}$ is the field of fractions of $\hat{\mathbf{k}}$.
\end{Example}

We denote the standard polynomial of degree $k$ by $\mathrm{st}_k$. By the Amitsur-Levitzki Theorem (see, for instance, \cite[Section 10.1.1]{AGPR}), $\mathrm{M}_n(\mathbf{k})$ satisfies $\mathrm{st}_{2n}$. Moreover, we shall describe the polynomial identities of $\mathcal{A}/\mathfrak{M}_n(\mathcal{A})$.

\begin{Prop}\label{prop1}
Let $\mathcal{A}$ be a $\mathbf{k}$-algebra. Then $\mathcal{A}/\mathfrak{M}_n(\mathcal{A})$ satisfies $\mathrm{st}_{2n}$ (and any other multilinear polynomial identity satisfied by $\mathrm{M}_n(\mathbf{k})$). Moreover, there exist a sequence $1\le n_1<n_2<\cdots<n_t\le n$ and commutative semiprime $\mathbf{k}$-algebras $\mathcal{C}_1$, \dots, $\mathcal{C}_t$ such that
$$
\mathrm{Id}(\mathcal{A}/\mathfrak{M}_n(\mathcal{A}))=\bigcap_{i=1}^t\mathrm{Id}(\mathrm{M}_{n_i}(\mathcal{C}_i)).
$$
In particular, if $\mathbf{k}=\mathbb{F}$ is an infinite field, then $\mathrm{Id}(\mathcal{A}/\mathfrak{M}_n(\mathcal{A}))=\mathrm{Id}(\mathrm{M}_{n'}(\mathbb{F}))$, for some $n'\le n$.
\end{Prop}

\begin{proof}
% Let $a_1, \dots, a_{2n} \in \mathcal{A}$ and set 
% \[
% b = \mathrm{st}_{2n}(a_1, \ldots, a_{2n}).
% \] 
% Suppose $b \notin \mathfrak{M}_n(\mathcal{A})$. Then there exists a prime homomorphic image $\mathcal{A}_0$ of $\mathcal{A}$ and an irreducible $Q(\mathcal{A}_0)$-module $\mathcal{M}$ with $d(\mathcal{M}) \le n$ such that $b \mathcal{M} \ne 0$. 

% However, we have 
% \[
% \mathrm{Id}(\mathrm{End}_\mathcal{D}(\mathcal{M})) = \mathrm{Id}(\mathrm{M}_{d(\mathcal{M})}(Z(\mathcal{D}))),
% \] 
% where $\mathcal{D} = \mathrm{End}_{Q(\mathcal{A}_0)}(\mathcal{M})$. Since $d(\mathcal{M}) \le n$, it follows that $\mathrm{End}_\mathcal{D}(\mathcal{M})$ satisfies $\mathrm{st}_{2n}$. Hence, $b$ is zero in $\mathrm{End}_\mathcal{D}(\mathcal{M})$, so $b \mathcal{M} = 0$, a contradiction. Therefore, $b \in \mathfrak{M}_n(\mathcal{A})$.

By \Cref{module_lemma}, the algebra $\mathcal{A}/\mathfrak{M}_n(\mathcal{A})$ is a subdirect product of prime algebras that are PI-equivalent to $\mathrm{M}_m(\mathbb{F})$, for some $m\le n$ and some $\mathbf{k}$-algebra $\mathbb{F}$ that is a field. Since each $\mathrm{M}_m(\mathbb{F})$ satisfies all the multilinear polynomial identities of $\mathrm{M}_n(\mathbf{k})$, the proof is complete.
\end{proof}

In particular, we obtain the following criterion for determining whether an Artinian algebra is PI:

\begin{Remark}\label{prop2}
Let $\mathcal{A}$ be an Artinian $\mathbf{k}$-algebra. Then $\mathcal{A}$ is PI if and only if
\[
\mathcal{B}(\mathcal{A}) = \mathfrak{M}_\infty(\mathcal{A}).
\]
Indeed, \Cref{prop19} tells us that the equality holds if $\mathcal{A}$ is PI. Conversely, since $\mathcal{A}$ is Artinian, $\mathcal{B}(\mathcal{A})$ is nilpotent. Moreover, $\mathfrak{M}_\infty(\mathcal{A}) = \mathfrak{M}_n(\mathcal{A})$ for some $n \in \mathbb{N}$. Hence, combining with \Cref{prop1}, $\mathcal{A}$ satisfies a power of the standard polynomial.
\end{Remark}

Note that a non-PI division ring $\mathcal{D}$ is Artinian, however $\mathcal{B}(\mathcal{D})\ne\mathfrak{M}_\infty(\mathcal{D})$. In addition, \Cref{UT} shows that the Artinian condition in the previous proposition is essential.

%The following examples illustrate two situations in which the previous proposition may fail if the Artinian condition is removed. Either $\mathcal{B}(\mathcal{A})$ is not nilpotent, or $\mathcal{A}$ admits an irreducible $Q(\mathcal{A}_0)$-module $\mathcal{M}$ with $d(\mathcal{M}) = \infty$.

% \begin{Example}
% Let $\mathcal{A}$ be a unital non-Artinian simple ring. Then, although $\mathcal{A}/\mathfrak{M}_n(\mathcal{A})$ is PI (it is the zero ring) for each $n \in \mathbb{N} \cup \{\infty\}$, the ring $\mathcal{A}$ itself is not PI.
% \end{Example}

An alternative description of the nonmatricial ideals is the following. %We denote by $\bigcup_nP_n$ the set of all multilinear polynomials.
\begin{Prop}\label{prop3}
The ideal $\mathfrak{M}_n(\mathcal{A})$ is the smallest semiprime ideal $K \subseteq \mathcal{A}$ such that $\mathcal{A}/K$ satisfies all multilinear polynomials identities of $\mathrm{M}_n(\mathbf{k})$.
\end{Prop}

\begin{proof}
From \Cref{module_lemma}, $\mathfrak{M}_n(\mathcal{A})$ is a semiprime ideal. First, let $P \subseteq \mathcal{A}$ be a prime ideal such that $\mathrm{Id}(\mathcal{A}/P)$ satisfies 
all multilinear polynomials identities of $\mathrm{Id}(\mathrm{M}_n(\mathbf{k}))$. Then $Q(\mathcal{A}/P)$ admits an irreducible module $\mathcal{M}$ with $d(\mathcal{M}) \le n$. In particular, any element of $\mathfrak{M}_n(\mathcal{A})$ annihilates $\mathcal{M}$. It means that $\mathfrak{M}_n(\mathcal{A})\subseteq P$.

Now, let $K = \bigcap_{P \in \Lambda} P$ be a semiprime ideal satisfying the condition that the multilinear polynomials in $\mathrm{Id}(\mathrm{M}_n(\mathbf{k}))$ are contained in $\mathrm{Id}(\mathcal{A}/K)$. Then $\mathcal{A}/K$ is a subdirect product of the algebras $\{\mathcal{A}/P\}_{P \in \Lambda}$. Each $\mathcal{A}/P$ satisfies all multilinear polynomial identities of $\mathrm{M}_n(\mathbf{k})$, and therefore each $P$ contains $\mathfrak{M}_n(\mathcal{A})$. It follows that $K$ contains $\mathfrak{M}_n(\mathcal{A})$ as well.
\end{proof}

\begin{Example}
Let $\mathbf{k} = \mathbb{F}$ be a field. We have $\mathrm{UT}_n(\mathbb{F}) \subseteq \mathrm{M}_n(\mathbb{F})$, so
\[
\mathrm{Id}(\mathrm{M}_n(\mathbb{F})) \subseteq \mathrm{Id}(\mathrm{UT}_n(\mathbb{F})).
\] 
However, since $\mathfrak{M}_n(\mathrm{UT}_n(\mathbb{F}))$ is the set of strictly upper triangular matrices, we have
\[
\mathfrak{M}_n(\mathrm{UT}_n(\mathbb{F})) \not\subseteq \{0\}.
\] 
Thus, the semiprimeness condition in \Cref{prop3} cannot be removed.
\end{Example}

As a consequence, we obtain the following characterization of the T-ideal of polynomial identities of matrix algebras when $\mathbf{k}=\mathbb{F}$ is a field of characteristic zero.
\begin{Prop}
Let $\mathbf{k}=\mathbb{F}$ be a field of characteristic zero and let $X$ be a countably infinite set. Then, for each $n \in \mathbb{N}$, one has
\[
\mathfrak{M}_n(\mathbb{F}\langle X \rangle) = \mathrm{Id}(\mathrm{M}_n(\mathbb{F})).
\] 
A similar result holds for an arbitrary set $X$.
\end{Prop}
\begin{proof}
By \Cref{prop1},
\[
\mathrm{Id}(\mathrm{M}_n(\mathbb{F})) \subseteq \mathrm{Id}(\mathbb{F}\langle X\rangle/\mathfrak{M}_n(\mathbb{F}\langle X\rangle)).
\]
Thus, $\mathrm{Id}(\mathrm{M}_n(\mathbb{F}))\subseteq\mathfrak{M}_n(\mathbb{F}\langle X\rangle)$.

Conversely, let $a\in\mathfrak{M}_n(\mathbb{F}\langle X\rangle)$. Then $a$ annihilates each irreducible $Q(\mathcal{A}_0)$-module $\mathcal{M}$ with $d(\mathcal{M})\le n$, where $\mathcal{A}_0$ is a prime homomorphic image of $\mathbb{F}\langle X\rangle$. Thus, $a$ is zero in each prime homomorphic image of $\mathbb{F}\langle X\rangle$ that satisfies all identities of $\mathrm{M}_n(\mathbb{F})$. In particular, $a$ is in the kernel of each homomorphism $\mathbb{F}\langle X\rangle\to\mathrm{M}_n(\mathbb{F})$. Hence, $a\in\mathrm{Id}(\mathrm{M}_n(\mathbb{F}))$.
\end{proof}

\section{On varieties not containing matrix algebras of higher order\label{sec5}}

Before we state our main result, recall the definition of complexity:
\begin{Def}[\cite{Belov}]\label{complexity}
A variety of $\mathbf{k}$-algebras $\mathscr{V}$ is said to have \emph{complexity $n\ge0$} if $n$ is the greatest integer such that   $\mathrm{M}_n(\mathbb{F})\in\overline{\mathscr{V}}$, for some simple homomorphic image $\mathbb{F}$ of $\mathbf{k}$, where $\overline{\mathscr{V}}$ is the variety generated by the homogeneous polynomial identities of $\mathscr{V}$.
\end{Def}

Finally, we obtain the following extension of \Cref{characterization_ringNoetherian}.

\begin{Thm}\label{characterization_nmatrix}
Let $\mathbf{k}$ be a unital commutative ring, $\mathscr{V}$ a variety of associative $\mathbf{k}$-algebras, and fix $n\in\mathbb{N}$. The following assertions are equivalent:
\begin{enumerate}
    \item Every simple $\mathbf{k}$-algebra in $\mathscr{V}$ has dimension over its center at most $(n-1)^2$;
    \item\label{iii.M_n} $\mathrm{M}_{n}(\mathbb{F}) \notin \mathscr{V}$ for each fraction field $\mathbb{F}$ of a prime homomorphic image of $\mathbf{k}$;
    \item[(ii')]\label{iii'.M_n} $\mathrm{M}_{n}(\mathbb{F}) \notin \mathscr{V}$ for each simple homomorphic image of $\mathbf{k}$;
    \item\label{iii.subdirectJ} For each $\mathcal{A} \in \mathscr{V}$, $\mathcal{A}/J(\mathcal{A})$ is a subdirect product of simple algebras with dimension over their center at most $(n-1)^2$;
    \item\label{iii.subdirect} For each $\mathcal{A} \in \mathscr{V}$, $\mathcal{A}/\mathcal{B}(\mathcal{A})$ is a subdirect product of prime algebras $\mathcal{P}$, where $\dim_C C\mathcal{P} \le (n-1)^2$ and $C$ is the extended centroid of $\mathcal{P}$;
    \item\label{iii.powercommutator} There exists $m\in\mathbb{N}$ such that $\mathrm{st}_{2(n-1)}(x_1,\ldots,x_{2n-2})^m \in \mathrm{Id}(\mathscr{V})$;
    \item\label{iii.jacobson} For every $\mathcal{A} \in \mathscr{V}$,
    \[
        \mathcal{B}(\mathcal{A}) = \{a \in \mathcal{A} \mid \text{$a$ is $(n-1)$-nonmatricial}\};
    \]
    \item\label{iii.jacobsonfg} For every finitely generated $\mathcal{A} \in \mathscr{V}$,
    \[
        \mathcal{B}(\mathcal{A}) = \{a \in \mathcal{A} \mid \text{$a$ is $(n-1)$-nonmatricial}\};
    \]
    \item\label{iii.someideal} For any $\mathcal{A} \in \mathscr{V}$ and $(n-1)$-nonmatricial element $a \in \mathcal{A}$, at least one of the following ideals is nil: $a\mathcal{A}$, $a\mathcal{A}+\mathbf{k}a$, $\mathcal{A}a$, $\mathcal{A}a+\mathbf{k}a$, $\mathcal{A}a\mathcal{A}$, $\mathcal{A}a\mathcal{A} + \mathcal{A}a + a\mathcal{A} + \mathbf{k}a$;
    \item\label{iii.allideals} For any $\mathcal{A} \in \mathscr{V}$ and $(n-1)$-nonmatricial element $a \in \mathcal{A}$, all of the above ideals are nil;
    \item\label{iii.sumnilelements} For every $\mathcal{A} \in \mathscr{V}$ and $(n-1)$-nonmatricial elements $a,b \in \mathcal{A}$, $a+b$ is nilpotent;
    \item For every $\mathcal{A} \in \mathscr{V}$, every $(n-1)$-nonmatricial element in $\mathcal{A}$ is nilpotent;
    \item For each $\mathcal{A}\in\mathscr{V}$ and elements $a_1,\dots,a_m \in \mathcal{A}$, if all products of elements in $\{a_1,\ldots,a_m\}$ of length at most $n-1$ is nilpotent, then the subalgebra generated by $a_1,\dots,a_m$ is nilpotent;
    \item The variety $\mathscr{V}$ has complexity $<n$.
    \newcounter{contador2}
    \setcounter{contador2}{\arabic{enumi}}
\end{enumerate}
\end{Thm}

\begin{proof}
$(i) \Rightarrow (ii)$: Clear.

The following implications follow analogously to the proof of \Cref{characterization_ringNoetherian}:
\[
(ii) \Rightarrow (i), \quad (ii) \Rightarrow (iv) \Rightarrow (v) \Rightarrow (ii).
\]

$(iv) \Rightarrow (vi)$: Let $a \in \mathcal{A}$ be $n$-nonmatricial. Assume $a \notin \mathcal{B}(\mathcal{A})$. Then there exists a prime homomorphic image $\mathcal{A}_0$ of $\mathcal{A}$ such that the image of $a$ is nonzero. From $(iv)$, each irreducible $Q(\mathcal{A}_0)$-module $\mathcal{M}$ satisfies $d(\mathcal{M}) \le n-1$. This contradicts the definition of $(n-1)$-nonmatricial. Hence,
\[
\{a \in \mathcal{A} \mid \text{$a$ is $(n-1)$-nonmatricial}\} \subseteq \mathcal{B}(\mathcal{A}).
\]

The converse follows from \Cref{module_lemma}.

$(vi) \Rightarrow (vii)$: Clear.

$(vii) \Rightarrow (ix)$: Same as in \Cref{characterization_ringNoetherian}.

$(ix) \Rightarrow (viii)$: Clear.

$(viii) \Rightarrow (ii)$: In $\mathrm{M}_{n}(\mathbb{F})$, the element $e_{11}$ is $(n-1)$-nonmatricial, but none of the ideals it generates is nil.

$(vi) \Rightarrow (x)$: Same as in \Cref{characterization_ringNoetherian}.

$(x) \Rightarrow (ii)$: In $\mathrm{M}_{n}(\mathbb{F})$, the elements $e_{12}$ and $e_{21}$ are $(n-1)$-nonmatricial, but $e_{12}+e_{21}$ is not nilpotent.

$(iv) \Rightarrow (iii)$: Each primitive PI algebra is simple with center equal to its extended centroid. Moreover, $\mathcal{A}/J(\mathcal{A})$ is a quotient of $\mathcal{A}/\mathcal{B}(\mathcal{A})$.

$(iii) \Rightarrow (ii)$: Clear.

$(vi) \Rightarrow (xi)$: Clear.

$(xi) \Rightarrow (ii)$: In $\mathrm{M}_{n}(\mathbb{F})$, every element is $(n-1)$-nonmatricial, but not every element is nilpotent.

$(ii)\Rightarrow(ii')\Rightarrow(xiii)$: It is clear from \Cref{complexity}, since $\overline{\mathscr{V}}\subseteq\mathscr{V}$.

$(xiii)\Rightarrow(xii)$: The result follows from \cite{Ufnarovskii} (see \cite[Theorem 1]{Belov} as well).

$(xii)\Rightarrow(ii)$: let $\mathbb{F}$ be the field of fraction of a prime homomorphic image of $\mathbf{k}$. Thus, the elements $e_{12}$, \dots, $e_{n-1,n}$, $e_{n1}\in\mathrm{M}_n(\mathbb{F})$ satisfy the hypothesis; however, it does not generate a nilpotent algebra. Thus, $\mathrm{M}_n(\mathbb{F})\notin\mathscr{V}$.
\end{proof}
\begin{Remark}
Consider the following statement:
\begin{enumerate}
\setcounter{enumi}{\arabic{contador2}}
\item\label{iii.commutators} For every finitely generated $\mathcal{A} \in \mathscr{V}$, there exists $m\in\mathbb{N}$ such that
    \[
       \mathrm{st}_{2(n-1)}(x_1,\ldots,x_{2n-2}) \cdots \mathrm{st}_{2(n-1)}(x_{2(m-1)(n-1)+1},\ldots,x_{2m(n-1)}) \in \mathrm{Id}(\mathcal{A});
    \]
\end{enumerate}
Similarly as in \Cref{remarknovo}, the above statement implies any of the statements of \Cref{characterization_nmatrix}. In addition, if $\mathbf{k}$ is Noetherian, then any of the statements of \Cref{characterization_nmatrix} implies the above statement. It is currently unknown whether (xiv) is equivalent to the others without the Noetherian assumption.
\end{Remark}

\begin{Def}
A variety $\mathscr{V}$ is said to be \emph{$n$-nonmatrix} if it satisfies any of the equivalent conditions of \Cref{characterization_nmatrix} for the number $n$.
\end{Def}
It is clear that if a variety is $n$-nonmatrix, then it is $m$-nonmatrix for each $m\ge n$. A variety is $1$-nonmatrix if and only all of its algebras are nil.

\begin{Remark}
By convention, the trivial variety is $0$-nonmatrix, and the variety of all associative algebras is $\infty$-nonmatrix. Thus, any variety of associative $\mathbf{k}$-algebras is $n$-nonmatrix for some $n \in \mathbb{N} \cup \{0,\infty\}$.  

Indeed, let $\mathscr{V}$ be a proper variety, and let $\mathcal{F}_\mathscr{V} := \mathbf{k}_\mathscr{V}\langle X \rangle \in \mathscr{V}$ be a free algebra generated by an infinite set of variables. Then, $\mathcal{F}_\mathscr{V}/\mathcal{B}(\mathcal{F}_\mathscr{V})$ is a subdirect product of prime PI-algebras. By Posner's Theorem, each of these algebras is PI-equivalent to a matrix algebra, and the sizes of these matrices are bounded by some maximum $n$. Hence, $\mathrm{st}_{2n}(x_1,\ldots,x_{2n}) \in \mathcal{B}(\mathcal{F}_\mathscr{V})$.  

Since every element of the Baer radical is nilpotent, there exists $m \in \mathbb{N}$ such that $(\mathrm{st}_{2n}(x_1,\ldots,x_{2n}))^m = 0$. Therefore, $\mathrm{st}_{2n}^m \in \mathrm{Id}(\mathscr{V})$, and from \emph{(vi)} of \Cref{characterization_nmatrix}, we conclude that $\mathscr{V}$ is an $(n+1)$-nonmatrix variety.
\end{Remark}

% \begin{Prop}
% Let $\mathbf{k}$ be a unital commutative ring and $\mathscr{V}$ a variety of associative $\mathbf{k}$-algebras. If $\mathscr{V}$ is a $n$-nonmatrix variety, then for each $\mathcal{A}\in\mathscr{V}$ and elements $a_1,\dots,a_m \in \mathcal{A}$, if all products of elements in $\{a_1,\ldots,a_m\}$ of length at most $n-1$ is nilpotent, then the subalgebra generated by $a_1,\dots,a_m$ is nilpotent. The converse holds.
% \end{Prop}
% \begin{proof}
% It follows from \cite[Theorem B]{Ufnarovskii}. For the converse, let $\mathbb{F}$ be the field of fraction of a prime homomorphic image of $\mathbf{k}$. Thus, the elements $e_{12}$, \dots, $e_{n-1,n}$, $e_{n1}\in\mathrm{M}_n(\mathbb{F})$ satisfy the hypothesis; however, it does not generate a nilpotent algebra.
% \end{proof}

An extension of \Cref{prop} to the context of $n$-nonmatrix varieties is thanks to a conjecture by I.P.~Shestakov, answered by V.A.~Ufnarovskii \cite{Ufnarovskii}, and then generalized by A.Ya.~Belov:

\begin{Prop}\label{Shestak}
Let $\mathbf{k}$ be a unital commutative ring and $\mathscr{V}$ a variety of associative $\mathbf{k}$-algebras. If $\mathscr{V}$ is a $n$-nonmatrix variety, then for each $\mathcal{A}\in\mathscr{V}$ and elements $a_1,\dots,a_m \in \mathcal{A}$, if all products of elements in $\{a_1,\ldots,a_m\}$ of length at most $n-1$ is integral, then the subalgebra generated by $a_1,\dots,a_m$ is $\mathbf{k}$-finitely generated. The converse holds if $\mathbf{k}$ contains an infinite prime homomorphic image.
\end{Prop}
\begin{proof}
It follows from \cite{C94} (see also \cite[Theorem 1]{Belov}). The converse is analogous to the converse of \Cref{prop} and the proof of (xii)$\Rightarrow$(ii) of \Cref{characterization_nmatrix}.
\end{proof}

\begin{Remark}
As in \Cref{remark16}, for a given variety $\mathscr{V}$ of associative $\mathbf{k}$-algebras, the following statements are equivalent:
\begin{enumerate}
\renewcommand{\labelenumi}{(\alph{enumi})}
    \item There exists $m \in \mathbb{N}$ such that $\mathscr{V}$ satisfies
    \[
        \mathrm{st}_{2n}(x_1,\ldots,x_{2n}) \cdots \mathrm{st}_{2n}(x_{2(m-1)n+1},\ldots,x_{2mn});
    \]
    \item For each $\mathcal{A} \in \mathscr{V}$, the ideal generated by
    \[
        \mathrm{st}_{2n}(\mathcal{A}) := \{\mathrm{st}_{2n}(a_1,\ldots,a_{2n}) \mid a_1,\ldots,a_{2n} \in \mathcal{A}\}
    \]
    is nilpotent;
    \item For each $\mathcal{A} \in \mathscr{V}$, the ideal generated by the image of $\mathrm{Id}(\mathrm{M}_n(\mathbf{k}))$ under all homomorphisms $\mathbf{k}\langle X\rangle\to\mathcal{A}$ in $\mathcal{A}$ is nilpotent.
\end{enumerate}

All of these statements hold if $\mathscr{V}$ is generated by a finitely generated algebra. Moreover, any of them implies that the Grassmann algebra does not belong to $\mathscr{V}$.

If $\mathbf{k}$ is a field of characteristic zero, then $\mathscr{V}$ is generated by a finitely generated algebra if and only if it does not contain the Grassmann algebra. Hence, these equivalences provide a criterion to characterize $n$-nonmatrix varieties that do not contain the Grassmann algebra over a field of characteristic zero.
\end{Remark}

\section*{Acknowledgments}
Part of this work was carried out while the first-named author was visiting the University of Bari. The author gratefully acknowledges the support and the stimulating academic atmosphere provided by the University of Bari.

The authors would like to express their gratitude to Professor Ivan Pavlovich Shestakov for valuable discussions.


\begin{thebibliography}{X}
\bibitem{AGPR} E.~Aljadeff, A.~Giambruno, C.~Procesi, A.~Regev, \emph{Rings with polynomial identities and finite dimensional representations of algebras}. American Mathematical Society Colloquium Publications 66. American Mathematical Society, Providence, RI, 2020.

\bibitem{Amitsur_Nullstellensatz} S.A.~Amitsur, \emph{A generalization of Hilbert's Nullstellensatz}, Proceedings of the American Mathematical Society \textbf{8} (1957), 649--656.

\bibitem{Belov} A.Ya.~Belov, \emph{On a Shirshov basis of relatively free algebras of complexity $n$} (Russian), Mat. Sb. \textbf{135} (1988), 373--384. Translation: Math. USSR Sb. \textbf{63} (1988),
363--374.
\bibitem{BRT97} Y.~Billig, D.~Riley, V.~Tasi\'c, \emph{Nonmatrix varieties and nil-generated algebras whose units satisfy a group identity}, Journal of Algebra \textbf{190} (1997), 241--252.
\bibitem{Bit} V.~Bittencourt, \emph{Variedades não matriciais em certas
classes de álgebras não associativas}. Tese de doutorado (IME-USP), 2016.
\bibitem{C} G.~\v{C}ekanu, \emph{On local finiteness in varieties of associative algebras}, Mathematics of the USSR-Sbornik \textbf{41} (1982), 181--201.
\bibitem{C94} G.~\v{C}ekanu, \emph{Independence and quasiregularity in algebras}, Doklady Akademii Nauk.~\textbf{337}. Russian Academy of Sciences, 1994.
\bibitem{GWbook} J.~Gardner, R.~Wiegandt, \emph{Radical theory of rings}. CRC Press, 2003.
\bibitem{Drensky_Formanek} V.~Drensky, E.~Formanek, \emph{Polynomial Identity Rings}, Adv. Courses Math. CRM Barcelona, Birkhäuser Verlag, Basel, 2004.
\bibitem{KBR} A.~Kanel-Belov, L.~Rowen, \emph{The Braun-Kemer-Razmyslov Theorem for affine PI-algebras}, Chebyshevskiĭ Sb.\textbf{21} (2020), 89--128.
\bibitem{K79} A.~Kemer, \emph{Nonmatrix varieties}, Algebra and Logic \textbf{19} (1980), 157--178.
\bibitem{Kemerbook} A.~Kemer, \emph{Ideals of identities of associative algebras}. Transl. Math. Monogr., \textbf{87}, American Mathematical Society, Providence, RI, 1991.
\bibitem{L65} V.~Latyshev, \emph{Generalization of the Hilbert Theorem on the finiteness of bases}, Siberian Mathematical Journal \textbf{7} (1966), 1112--1113.
\bibitem{LI} V.~Latyshev, \emph{Complexity of nonmatrix varieties of associative algebras. I}, Algebra and Logic \textbf{16} (1977), 98--122.
\bibitem{LII} V.~Latyshev, \emph{Complexity of nonmatrix varieties of associative algebras. II}, Algebra and Logic \textbf{16} (1977), 122--133.
%\bibitem{L80} V.~Latyshev, \emph{Nonmatrix varieties of associative algebras, Matematicheskie Zametki \textbf{27} (1980), 147--156.}
\bibitem{MPR} S.~Mishchenko, V.~Petrogradsky, A.~Regev, \emph{Characterization of nonmatrix varieties of associative algebras}, Israel Journal of Mathematics \textbf{182} (2011), 337--348.
\bibitem{SB} I.~Shestakov, V.~Bittencourt, \emph{Nonmatrix varieties of nonassociative algebras}, Algebra and Logic, \textbf{62} (6) (2024), 532--547.
\bibitem{Ufnarovskii} V.A.~Ufnarovskii, \emph{An independence theorem and its consequences} (Russian), Mat. Sb. \textbf{128}, (1985) 124--132. Translation: Math. USSR Sb. \textbf{56} (1985), 121--129.
\end{thebibliography}
\end{document}